\documentclass[12pt]{amsart}

\usepackage[dvips]{graphicx}
	\usepackage{amsmath,amssymb,amsthm,wrapfig,amsfonts,enumerate,latexsym}

\usepackage{comment,color}

\newtheorem{theorem}{Theorem}[section]

\newtheorem{corollary}[theorem]{Corollary}
\newtheorem{proposition}[theorem]{Proposition}

\theoremstyle{definition}

\newtheorem{definition}[theorem]{Definition}
\newtheorem{remark}[theorem]{Remark}
\newtheorem{condition*}[theorem]{Condition}
\newtheorem*{acknowledgement}{Acknowledgements}       
 
\numberwithin{equation}{section}

\begin{document}

\title[Uniform estimates of nonlinear spectral gaps]{Uniform estimates of nonlinear spectral gaps}
\author
[T. Kondo]{Takefumi Kondo}
\author
[T. Toyoda]{Tetsu Toyoda}

\address
[Takefumi Kondo]
{\endgraf Graduate School of Mathematics, 
\endgraf
Nagoya University, Chikusa-ku, Nagoya, 464-8602, Japan}
\email
{kondo.takefumi@math.nagoya-u.ac.jp}

\address
[Tetsu Toyoda]
{\endgraf Suzuka National College of Technology, 
\endgraf
Shiroko-cho, Suzuka, Mie, 510-0294, Japan}
\email
{toyoda@genl.suzuka-ct.ac.jp}

\date{}


\begin{abstract}
By generalizing the path method, we show that nonlinear spectral gaps of a finite connected graph 
are uniformly bounded from below by a positive constant which is independent 
of the target metric space. 
We apply our result to an $r$-ball $T_{d,r}$ in the $d$-regular tree, and observe 
that the asymptotic behavior 
of nonlinear spectral gaps of $T_{d,r}$ as $r\to\infty$
does not depend on the target metric space, 
which is in contrast to the case 
of a sequence of expanders. 
We also apply our result to the $n$-dimensional Hamming cube $H_n$ 
and obtain an estimate of its nonlinear spectral gap with respect to 
an arbitrary metric space, which is asymptotically sharp as $n\to\infty$. 
\end{abstract}

\maketitle

\section{Introduction}\label{into-sec}

Let $G=(V,E)$ be a graph, whose set 
of vertices and unoriented edges are denoted by $V$ and $E$, respectively. 
In this paper, graphs are always assumed to be simple, 
connected and finite. 
We denote by $\vec{E}$ the set $\{ (x,y)\in V\times V\hspace{1mm}|\hspace{1mm}\{ x,y\}\in E\}$.  
A {\em weight function} on $G$ is a symmetric function 
$m :V\times V\to\lbrack 0,\infty )$ 
whose support equals $\vec{E}$. 
The pair $(G,m)$ is called a {\em weighted graph}. 
A weight function $m$ induces a weight 
$m(x)$ of each vertex $x\in V$ by $m(x)=\sum_{y\in V}m(x,y)$. 
We use the convention that $m(\emptyset )=\sum_{x\in V}m(x)$. 
We call the following special weight function $m$ the {\em uniform} weight function on $G$:  
$$
m(x,y)=
\begin{cases}
1,&\quad \textrm{if}\quad (x,y)\in \vec{E}, \\
0,&\quad \textrm{otherwise} .
\end{cases}
$$

Throughout this paper, 
graphs and metric spaces are always assumed to contain at least two distinct points. 

\begin{definition}
Let $(G,m)$ be a weighted graph, 
and $(X,d_X)$ be a metric space. 
We define the {\em nonlinear spectral gap} $\lambda (G,X)$ 
of $G$ with respect to $X$ to be the reciprocal of the smallest constant $C >0$
such that the following Poincar\'e inequality holds for any $f: V\to X$: 
$$
\frac{1}{m (\emptyset )}\sum_{x,y\in V}m(x) m(y) d_X (f(x),f(y))^2
\leq
C
\sum_{x,y\in V}m(x,y)d_X (f(x),f(y))^2 .
$$
\end{definition}
Although the notation 
$\lambda_1^{\mathrm{Gro}}(G,X)$ was used for $\lambda(G,X)$ in the first author's previous paper \cite{K2}, 
we use this notation to avoid the notational complexity. 
We will see later that the constant $C>0$ in the definition always exists. 
Hence nonlinear spectral gaps are always positive real numbers. 
If $X=\mathbb{R}$, 
$\lambda (G, \mathbb{R})$ equals 
the first positive eigenvalue $\mu_1 (G)$ of the combinatorial Laplacian $\Delta$ on $G$, 
which acts on a real-valued function $f$ on $V$ as 
$$
\Delta f (x)
=
f(x)-
\sum_{y\in V}\frac{m(x,y)}{m(x)}f(y), \quad 
x\in V. 
$$

Nonlinear spectral gaps play important roles both 
in geometric group theory and metric geometry. 
For example, they are related 
to the fixed point property of discrete groups (\cite{G}, \cite{IKN}, \cite{IKN2}, \cite{IN},\cite{NS}, \cite{Wa}), measure concentration of metric measure 
spaces (\cite{G2}) and coarse embeddability of a metric space into another metric space (\cite{G}, \cite{MN}, \cite{MN2}). 

Let us see the difference between $\lambda (G,X)$ and 
$\mu_{1}(G)=\lambda (G,{\mathbb{R}})$. 
For a Hilbert space $\mathcal{H}$, 
by summing over the coordinates with respect to some orthonormal basis, 
it is straightforward to see that 
$\lambda (G, \mathcal{H})=\lambda (G, \mathbb{R})=\mu_1 (G)$ 
for any graph $G$. 
On the other hand, 
as we see in Proposition \ref{degree-bounded-decay}, 
the asymptotic behavior of 
nonlinear spectral gaps of a sequence of expanders changes 
drastically if the target metric space changes. 
A sequence of (uniformly weighted) finite connected graphs $\left\{ G_n =(V_n ,E_n )\right\}_{n=1}^{\infty}$ 
is called a {\em sequence of expanders} if it satisfies the following properties: 
\begin{enumerate}
\item
$\lim_{n\to\infty}| V_n |=\infty$.
\item
There exists $d$ such that $\mathrm{deg}(v)\le d$ for all $v\in V_n$ 
and all $n$. 
\item
There exists $\lambda >0$ such that 
$
\mu_1 (G_n )\ge\lambda
$
for all $n$.  
\end{enumerate}

\begin{proposition}\label{degree-bounded-decay}
Suppose that a sequence of 
(uniformly weighted) graphs $\{G_n =(V_n ,E_n )\}_{n=1}^{\infty}$ 
satisfies the above properties $\mathrm{(1)}$ and $\mathrm{(2)}$. 
Then there exists a metric space $X$ such that 
$$
\lambda(G_n ,X)\lesssim_d \frac{1}{(\log |V_n |)^2} . 
$$
\end{proposition}

Here, 
$A\lesssim B$ or $B\gtrsim A$ means that there exists a 
universal constant $C>0$ such that $A\leq CB$. 
We write $A\asymp B$ 
if both $A\lesssim B$ and $B\lesssim A$ hold. 
If we have $A\leq C_p B$ for 
a constant $C_p >0$ which depends only on some parameter $p$, we write 
$A\lesssim_p B$ or $B\gtrsim_p A$. 
We write $A\asymp_p B$ 
if both $A\lesssim_p B$ and $B\lesssim_p A$ hold. 

In particular, Proposition \ref{degree-bounded-decay} shows that for any sequence of expanders $\{ G_n\}$, 
for some metric space $X$,
we have 
\begin{equation}\label{expander-est}
\frac{\lambda(G_n ,X)}{\mu_1 (G_n )}\lesssim_d \frac{1}{(\log |V_n |)^2} .
\end{equation}
In fact, this estimate of the ratio between linear and nonlinear spectral gaps turns out to be a sharp order of magnitude 
by a simple application of Bourgain's embedding theorem: 
\begin{equation}\label{bourgain-est}
\frac{\lambda(G,X)}{\mu_1 (G)}
\gtrsim
\frac{1}{(\log |V|)^2} 
\end{equation}
(see Proposition \ref{bourgain-estimate}).

A purpose of this paper is to establish a more accurate lower estimate of the nonlinear spectral gap of a given graph, 
which is independent of target metric spaces, 
by generalizing the path-method developed by Jerrum and Sinclair \cite{JS},
Diaconis and Stroock \cite{DSt}, and Quastel \cite{Q}, and 
Diaconis and Saloff-Coste \cite{DSa}
(see also Saloff-Coste \cite{SC}). 
Although the original path method is well-known 
in the context of random walks as one giving  
only rough estimates of the linear spectral gap, 
we will see that 
our generalization gives asymptotically sharp estimates of the nonlinear spectral gaps in some examples.

A {\em path} on a graph $G=(V,E)$ is a finite sequence 
$$
(x_1 ,x_2 ), (x_2 ,x_3),\ldots ,(x_{n-1} ,x_{n} ), (x_n ,x_{n+1} )
$$
in $\vec{E}$ such that all the vertices $x_1, \ldots ,x_{n+1}$ are distinct. 
For two distinct vertices $x,y\in V$, we denote by $\Gamma (x,y)$ 
the set of all paths from $x$ to $y$. 
For a weight function $w: V\times V\to\lbrack 0,\infty )$ 
and a path 
$\gamma : (x_1 ,x_2 ), (x_2 ,x_3),\ldots ,(x_n ,x_{n+1} )$ on $G$, 
we define 
$$
| \gamma |_w
=
\sum_{i=1}^{n}\frac{1}{w(x_i ,x_{i+1} )} .
$$
\begin{theorem}\label{nonlinearSC-th}
Let $(G,m)$ be a weighted graph. 
Let $w$ be another weight function on $G$. 
We assign 
one path $\gamma (x,y)\in\Gamma (x,y)$ 
to each ordered pair $(x, y) \in V\times V$ of distinct vertices. 
We define 
$$
A(w)
=
\max_{e\in\vec{E}}\left\{
\frac{1}{m(\emptyset )}\frac{1}{m(e)}w(e)
\sum_{(x,y)\hspace{1mm}\textrm{s.t.}\hspace{1mm}\gamma (x,y)\owns e}
|\gamma (x,y)|_w m(x)m(y)\right\} . 
$$
Then we have 
$$
\lambda (G, X)\geq\frac{1}{A(w)}
$$
for any metric space $X$. 
\end{theorem}

For the case $X=\mathbb{R}$, Theorem \ref{nonlinearSC-th} is the usual path method. 

In Section \ref{tree-sec}, we apply Theorem \ref{nonlinearSC-th} to 
the $n$-dimensional Hamming cube $H_n$ and 
an $r$-balls $T_{d,r}$ in the $d$-regular tree both equipped with uniform weights. 

For Hamming cubes, 
as a corollary of Theorem \ref{nonlinearSC-th}, 
we obtain 
\begin{equation}\label{hamming-asymptotic}
\lambda (H_n, X)\gtrsim\frac{1}{n^2}
\end{equation}
for any metric space $X$. 
Since it is known that 
$$
\mu_1 (H_n )=\frac{2}{n}, 
$$
this estimate is not asymptotically sharp for $X=\mathbb{R}$. 
However, 
as we see in Proposition \ref{hamming-prop} this is asymptotically sharp 
as an estimate for arbitrary metric spaces. 
Note that the estimate \eqref{bourgain-est} only yields 
$$
\frac{\lambda (H_n,X)}{\mu_1 (H_n)}\gtrsim \frac{1}{n^2}.
$$

As an another application of Theorem \ref{nonlinearSC-th}, 
we show that 
\begin{equation}\label{tree-asymptotic}
\lambda(T_{d,r},X)\asymp_d \frac{1}{(d-1)^r}
\end{equation}
for any metric space $X$ (Corollary \ref{tree-order}). 
Hence, 
the asymptotic behavior 
of nonlinear spectral gaps of $T_{d,r}$ as $r\to\infty$ 
does not depend upon the target metric space in contrast to 
the case of a sequence of expanders \eqref{expander-est}. 
This implies that 
$$
\frac{\lambda (T_{d,r},X)}{\mu_1 (T_{d,r})}\asymp_d 1,
$$
which is much better than the estimate \eqref{bourgain-est}. 

Finally, we briefly review 
some results concerning the estimates of 
nonlinear spectral gaps with respect to the class of metric spaces called $\mathrm{CAT}(0)$ spaces. 
In order to estimate $\lambda(G,X)$ from below, 
Izeki and Nayatani \cite{IN} 
introduced an invariant $0\leq\delta (X)\leq 1$ for a complete $\mathrm{CAT}(0)$ space $X$, 
and showed that 
$$
\frac{1}{2}
(1-\delta (X))\mu_1 (G)
\leq
\lambda(G,X) 
\leq\mu_1 (G) .
$$
Hence, if 
$\delta (X)<1$ we have 
$$
\mu_1 (G_n )\asymp \lambda (G_n,X) 
$$ 
for any sequence of graphs $\{ G_n \}_{n=1}^{\infty}$. 
This implies in particular that a sequence of expanders in a usual sense is also 
a sequence of expanders with respect to $X$. 

Many estimates of $\delta$ have been done up to now (see \cite{IN}, \cite{IKN2}, \cite{FT}, \cite{To2}, and \cite{To3}). 
According to these estimates, 
the class of $\mathrm{CAT}(0)$ spaces with $\delta <1$ includes 
Hilbert spaces, Hadamard manifolds, trees, complete $\mathrm{CAT}(0)$ cube complexes, 
and geodesically complete $\mathrm{CAT}(0)$ spaces which admit proper cocompact isometric group actions such as 
Bruhat-Tits buildings associated to semi-simple algebraic groups. 
On the other hand, the first author \cite{K2} proved 
the existence of a complete $\mathrm{CAT}(0)$ space $X$ and 
a sequence $\{G_n \}_{n=1}^{\infty}$ of expanders such that
\begin{equation}\label{kondo-expander}
\lim_{n\to\infty}\lambda(G_n,X) =0. 
\end{equation}
In particular, such a $\mathrm{CAT}(0)$ space $X$ satisfies $\delta (X)=1$. 
This result means that 
a drastic change in 
the nonlinear spectral gap 
may happen 
even within the class of $\mathrm{CAT}(0)$ spaces. 

Related estimates of the nonlinear spectral gap in comparison with the linear spectral gap are 
also found in Naor-Silberman (\cite{NS}). 
They showed that if a metric space $X$ has a finite Nagata dimension and $G$ is any graph, 
$$
\frac{\lambda(G ,X)}{\mu_1 (G)^2}\gtrsim_X 1 
$$
holds. 

The paper is organized as follows. 
In Section \ref{SC-sec}, we develop techniques to 
estimate nonlinear spectral gaps uniformly from below. 
First, we present a simple method to estimate nonlinear spectral gaps 
by using estimates of Euclidean distortions. 
Then we prove Theorem \ref{nonlinearSC-th}. 
In Section \ref{tree-sec}, we apply Theorem \ref{nonlinearSC-th} 
to Hamming cubes and trees and obtain \eqref{hamming-asymptotic} and \eqref{tree-asymptotic}.

\section{Uniform lower bounds of nonlinear spectral gaps}\label{SC-sec}

In this section, we prove Theorem \ref{nonlinearSC-th}. 
Before proving it, 
we present one simple argument to obtain a uniform lower bound of non-linear spectral 
gaps by using embeddings into a Hilbert space. 
This type of estimation is also found in Naor and Silberman \cite{NS}. 

\begin{definition}
Let $(X,d_X)$ and $(Y,d_Y)$ be metric spaces. 
For an injective mapping $f: X\to Y$, the {\em distortion} of $f$  
is defined to be the product  
$$
\sup_{x,y\in X, \hspace{1mm}x\neq y}\frac{d_Y (f(x),f(y))}{d_X (x,y)}
\cdot
\sup_{x,y\in X, \hspace{1mm}x\neq y}\frac{d_X (x,y)}{d_Y (f(x),f(y))} .
$$
If a mapping $f: X\to Y$ is not injective, the distortion of $f$ is defined to be $\infty$. 
The {\em distortion} $c_Y (X)$ is the 
infimum of the distortions of all mappings from $X$ to $Y$. 
We denote by $c_p (X)$ the distortion of $X$ into $L_p$. 
\end{definition}

\begin{proposition}\label{izeki-estimate}
Let $(G,m)$ be a finite connected weighted graph with $n\geq 2$ vertices, and 
let $(X,d_X )$ be a metric space. 
We define 
$$
c_2 (X,n)
=
\max\{c_2 (X')\hspace{1mm}|\hspace{1mm} X'\subset X, |X' |\leq n\}.
$$
Then we have 
$$
\lambda(G,X)
\geq
\frac{1}{c_2 (X,n)^2}\mu_1 (G). 
$$
\end{proposition}

\begin{proof}
Let $f:V\to X$ be an arbitrary mapping. 
Then, by the definition of $c_2 (X,n)$, for any $\varepsilon >0$, there exists a mapping 
$\varphi :f(V)\to\ell_{2}$ such that  
$$
d_X (f(x),f(y))\leq\|\varphi\circ f(x)-\varphi\circ f (y)\|\leq  (c_2 (X,n)+\varepsilon ) d_X (f(x),f(y)) 
$$
holds for any $x,y\in V$. 
Hence, we have 
\begin{align*}
\frac{1}{m (\emptyset )}\sum_{x,y\in V}m(x)& m(y) d_X ( f(x),f(y))^2 \\
&\leq
\frac{1}{m (\emptyset )}\sum_{x,y\in V}m(x) m(y) \| \varphi\circ f(x)-\varphi\circ f(y)\|^2 \\
&\leq
\frac{1}{\mu_1 (G)}
\sum_{x,y\in V}m(x,y)\| \varphi\circ f(x)-\varphi\circ f(y)\|^2 \\
&\leq
\frac{(c_2 (X,n)+\varepsilon )^2}{\mu_1 (G)}
\sum_{x,y\in V}m(x,y)d_X ( f(x), f(y))^2 .
\end{align*}
Since $\varepsilon >0$ is arbitrary, this proves the proposition. 
\end{proof}

Combining this proposition with the following 
well-known Bourgain's embedding theorem, 
we obtain a uniform estimate on nonlinear spectral gaps. 

\begin{theorem}[Bourgain \cite{B}]
For every $n$-point metric space $(X,d_X )$, we have 
$$
c_2 (X)\lesssim\log n.
$$ 
\end{theorem}

\begin{proposition}\label{bourgain-estimate}
Let $(G,m)$ be a connected weighted graph with $n$ vertices, and 
let $(X,d_X )$ be a metric space. 
Then we have 
$$
\frac{\lambda(G,X)}{\mu_1 (G)}
\gtrsim
\frac{1}{(\log n)^2}. 
$$
\end{proposition}

Though we do not see the graph structure in the proof of Proposition \ref{bourgain-estimate}, 
this is asymptotically sharp 
as is shown by a sequence $\{ G_n \}_{n=1}^{\infty}$ of expanders, 
for which we have 
$\mu_1 (G_n ) \asymp 1$ and 
$\lambda(G_n ,X)\lesssim_d (\log |G_n| )^{-2}$ 
for some metric space $X$ by Proposition \ref{degree-bounded-decay}. 
Here we give 
the proof of Proposition~1.2 in Introduction. 
\begin{proof}[Proof of Proposition \ref{degree-bounded-decay}]
We assume that $d\geq 3$ since 
when $d = 2$, a graph $G_n$ is a path graph, and in this case 
the proposition follows from Remark \ref{path-graph-remark}. 
Let $(X,d_X )$ be a metric space containing all $G_n$ isometrically, and 
let $f_n: V_n \to X$ be an isometric embedding for each $n$. 
Since any $r$-ball in $G_n$ contains at most 
$\sum_{i=0}^{i=r}(d-1)^i \asymp_d (d-1)^r$ vertices, 
at least $|V_n | /2$ vertices $y\in V_n$ satisfy 
$d_X (x,y)\gtrsim_d \log |V_n |$ for any $x\in V_n$. 
Thus, at least half of the pairs $(x,y)\in V_n \times V_n$ satisfy 
$$
d_X (f_n (x), f_n (y) )\gtrsim_d \log |V_n |.
$$
Hence, we have 
\begin{align*}
\frac{\sum_{x,y\in V_n}m(x,y) d_X (f_n (x),f_n (y))^2}{\frac{1}{m (\emptyset )}
\sum_{x,y\in V_n}m(x) m(y) d_X (f_n (x),f_n (y))^2}
&\leq
\frac{d^2|V_n |^2}{\sum_{x,y\in V_n}d_X (f_n (x),f_n (y))^2} \\
&\lesssim_d
\frac{1}{(\log |V_n |)^2}, 
\end{align*}
which proves the proposition. 
\end{proof}

However, as we will see later, for a specific sequence of graphs, 
we can obtain a more accurate estimate 
by generalizing the path method. 

\begin{proof}[Proof of Theorem \ref{nonlinearSC-th}]
Let $f:V\to X$ be a map. 
Then the triangle inequality and the Cauchy-Schwarz inequality yield the following.  
\begin{align*}
d_X \left(f(x),f(y)\right)^2
&\leq
\left\{ \sum_{(u,v)\in\gamma (x,y)}d_X (f(u),f(v)) \right\}^2 \\
&\leq
\left( \sum_{(u,v)\in\gamma (x,y)}w(u,v)^{-1}\right)
\left( \sum_{(u,v)\in\gamma (x,y)} d_X (f(u),f(v))^2 w(u,v) \right) \\
&=
|\gamma (x,y)|_w
\sum_{(u,v)\in\gamma (x,y)}d_X (f(u),f(v))^2 w(u,v). 
\end{align*}
Multiplying $\frac{m(x)m(y)}{m(\emptyset )}$ both sides of the above inequality and 
summing over all $(x,y)\in V\times V$, we obtain the following. 

\begin{align*}
\frac{1}{m(\emptyset )}&
\sum_{(x,y)\in V\times V}
m(x) m(y)d\left(f(x),f(y)\right)^2 \\
&\leq
\sum_{(x,y)\in V\times V}
\sum_{(u,v)\in\gamma (x,y)}
\frac{1}{m(\emptyset )} |\gamma (x,y)|_w
m(x)m(y)w(u,v)d(f(u),f(v))^2 \\
&=
\sum_{(u,v)\in\vec{E}}
\sum_{(x,y)\hspace{0.2mm}\textrm{s.t.}\hspace{0.2mm}\gamma (x,y)\owns (u,v)}
\frac{1}{m(\emptyset )}
|\gamma (x,y)|_w
m(x)m(y)w(u,v)d(f(u),f(v))^2 \\
&=
\sum_{(u,v)\in\vec{E}}\Bigg\lbrack m(u,v) d(f(u),f(v))^2 \\
&\hspace{2cm}
\times\bigg\{\frac{1}{m(\emptyset )}\frac{1}{m(u,v)}w(u,v)
\sum_{(x,y)\hspace{0.2mm}\textrm{s.t.}\hspace{0.2mm}\gamma (x,y)\owns (u,v)}|\gamma (x,y)|_w
m(x)m(y)\bigg\}\Bigg\rbrack
\end{align*}
Thus, 
\begin{align*}
\frac{1}{m(\emptyset )}&
\sum_{(x,y)\in V\times V}
m(x) m(y)d\left(f(x),f(y)\right)^2 \\
&\leq
\max_{(u,v)\in\vec{E}}\bigg\{\frac{1}{m(\emptyset )}\frac{1}{m(u,v)}w(u,v)
\sum_{(x,y)\hspace{0.2mm}\textrm{s.t.}\hspace{0.2mm}\gamma (x,y)\owns (u,v)}|\gamma (x,y)|_w
m(x)m(y)\bigg\} \\
&\hspace{2cm}
\times\sum_{(u,v)\in\vec{E}}m(u,v) d(f(u),f(v))^2 \\
&=
A(w)\sum_{(u,v)\in\vec{E}}m(u,v)d(f(u),f(v))^2 ,
\end{align*}
which proves the theorem. 
\end{proof}

\section{Nonlinear spectral gaps of Hamming cubes and trees}\label{tree-sec}

In this section, we applied Theorem \ref{nonlinearSC-th} to Hamming cubes and trees. 
Let $H_n$ be the $n$-dimensional Hamming cube equipped with the uniform weight. 
In \cite{SC}, it was shown that 
\begin{equation*}\label{hamming-linear}
\mu_1 (H_n )\gtrsim\frac{1}{n^2} 
\end{equation*}
by using the path method. 
However, it is not asymptotically sharp 
since it is known that 
\begin{equation*}\label{hamming-linear-optimal}
\mu_1 (H_n )=\frac{2}{n}.
\end{equation*}
On the other hand, Theorem \ref{nonlinearSC-th} guarantees that 
the estimation in \cite{SC} also works for nonlinear spectral gaps 
with respect to any target metric spaces. 
Thus we actually have 
\begin{equation}\label{hamming-nonlinear}
\lambda (H_n ,X)\gtrsim\frac{1}{n^2} 
\end{equation}
for an arbitrary metric space $X$. 
This is the right order of magnitude for general metric spaces 
because if 
we take the identity mappings $\iota_n: H_n\to H_n$, 
we see that 
\begin{equation}\label{hamming-eq}
\lambda(H_n ,H_n )
\leq
\frac{\sum_{x,y\in V}m(x,y)d(\iota_n (x),\iota_n (y))^2}{
\frac{1}{m (\emptyset )}\sum_{x,y\in V}m(x) m(y) d(\iota_n (x),\iota_n (y))^2}
=
\frac{4}{n(n+1)}. 
\end{equation}
\begin{proposition}\label{hamming-prop}
Let $H_n$ be the $n$-dimensional Hamming cube equipped with the uniform weight. 
Let $X$ be an arbitrary metric space. 
Then, we have 
$$
\lambda (H_n , X)\gtrsim\frac{1}{n^2},  
$$
and this estimate is asymptotically sharp. 
\end{proposition}
In fact, we will show in a forthcoming paper \cite{KT3} that the right-hand side of \eqref{hamming-eq} gives the optimal value of the infimum of 
the nonlinear spectral gaps:
$$
\inf\{\lambda(H_n ,X)\hspace{1mm}:\hspace{1mm}X \textrm{ is a metric space}\}
=
\lambda(H_n ,H_n )
=
\frac{4}{n(n+1)}. 
$$ 

Now, we proceed to another application of Theorem \ref{nonlinearSC-th}. 
For an integer $d\geq 2$, 
let $T_d$ be the $d$-regular tree. 
For an integer $r\geq 0$, let $T_{d,r}$ be an $r$-ball in $T_d$. 
We consider the uniform weight function $m$ on $T_{d,r}$. 
Combination of Bourgain's estimate (\cite{B2}) of the Euclidean distortion 
$$
c_2 (T)\asymp\sqrt{\log\log n}
$$
for an $n$-point tree $T$ 
with 
Proposition \ref{izeki-estimate} 
only yields that 
$$
\frac{\lambda (T,X)}{\mu_1 (T)}
\gtrsim
\frac{1}{\log\log n}
$$
for an arbitrary metric space $X$, 
which implies that for $T_{d,r}$, we have 
\begin{equation}\label{log-ineq}
\frac{\lambda(T_{d,r},X)}{\mu_1 (T_{d,r})}
\gtrsim_d 
\frac{1}{\log r}.
\end{equation}
Unfortunately, the right hand side of \eqref{log-ineq} 
tends to $0$ as $r$ goes to infinity. 
The following proposition gives a more accurate estimate. 

\begin{proposition}\label{from-below-prop}
Let $(X,d)$ be any metric space. 
Then we have 
$$
\lambda(T_{d,r}, X)
\geq
\frac{d-2}{d^2 (d-1)}
\frac{\left\{ (d-1)^r -1\right\}}{(d-1)^r}
\times
(d-1)^{-r}
$$
for any $d\geq 3$ and $r\geq 1$. 
\end{proposition}

\begin{proof}
Since $T_{d,r}$ is a tree, 
each $\Gamma (x,y)$ contains only one path $\gamma (x,y)$. 
We define a weight function $w: V\times V\to (0,\infty)$ on $T_{d,r}$ 
by setting 
$$
w(x,y)=(d-1)^{i+1}  ,\quad (x,y)\in\vec{E}, 
$$
where $i$ is the graph distance from the center vertex $o$ to 
$\{x ,y\}$. 

Let $e=(u,v)\in\vec{E}$ be an arbitrary ordered edge. 
Then $e$ separates $V$ into two connected components, 
$U$ containing $u$ and 
$W$ containing $v$. 
According to Theorem \ref{nonlinearSC-th}, we need to estimate 
$$
A(w,e)
=
\frac{1}{m(\emptyset )}\frac{1}{m(e)}w(e)
\sum_{(x,y)\hspace{1mm}\textrm{s.t.}\hspace{1mm}\gamma (x,y)\owns e}
|\gamma (x,y)|_w m(x)m(y)
$$
from above. 
We can assume that $o\in U$. 
Let $k$ be the graph distance between $o$ and $v$. 
Then we have 
\begin{align*}
|W|
&=
\sum_{l=0}^{r-k}(d-1)^{l}
=
\frac{(d-1)^{r-k+1}-1}{d-2}
\leq
(d-1)^{r-k+1},
\\
|U|
&\leq
|V|=
1+\sum_{l=1}^r d(d-1)^{l-1}
\leq
\frac{d}{d-2}(d-1)^r , \\
m(\emptyset )
&=
|\vec{E}|
=2\sum_{l=1}^r d(d-1)^{l-1}
=\frac{2d}{d-2}\left\{ (d-1)^r -1\right\} ,\\
m(v)&\leq d, \\
|\gamma (x,y)|_w 
&\leq
2
\sum_{i=1}^r \frac{1}{(d-1)^i}
\leq
\frac{2}{d-2}
\end{align*}
for every $v,x,y\in V$. 
Thus, 
\begin{align*}
A(w,e) 
&\leq
\frac{d-2}{2d\left\{ (d-1)^r -1\right\}}
\cdot\frac{1}{1}\cdot (d-1)^k \cdot
\sum_{(x,y)\hspace{1mm}\textrm{s.t.}\hspace{1mm}\gamma (x,y)\owns e}
\frac{2}{d-2}
d^2
\\
&\leq
\frac{d-2}{d\left\{ (d-1)^r -1\right\}}
\cdot (d-1)^k
\cdot
(d-1)^{r-k+1} 
\frac{d}{d-2}(d-1)^r
\frac{1}{d-2}
d^2
\\
&=
\frac{d^2 (d-1)}{d-2}
\frac{(d-1)^r}{\left\{ (d-1)^r -1\right\}}
\cdot
(d-1)^r ,
\end{align*}
which proves the proposition. 
\end{proof}

\begin{proposition}\label{from-above-prop}
Let $(X,d)$ be any metric space. 
Then we have 
\begin{multline*}
\lambda(T_{d,r},X) \\
\leq
2\bigg(
\frac{d\left\{ (d-1)^{r-1}-1\right\}}{d-2}
+(d-1)^{r-1} \hspace{5cm}\\
+\frac{\left\{ (d-1)^{r-1}-1\right\} (d-1)^r}{(d-1)^{r}-1}
+\frac{(d-2)\left\{(d-1)^{2r-1}\right\}}{d\left\{ (d-1)^{r}-1\right\}}\bigg)^{-1}
\end{multline*}
for any $d\geq 3$ and $r\geq 1$.  
\end{proposition}

\begin{proof}
We take an edge $e=\{ o,v\}\in E$ containing the center. 
The edge $e$ divides $V$ into two components $U$ containing $o$ and $W=U^c$. 
Let $f: V\to X$ be a mapping sending 
$U$ to $p$ and $W$ to $q$, where 
$p,q\in X$ are distinct points. 

The weight of each vertex is either $d$ or $1$, and 
we have 
\begin{align*}
|\{ u\in U\hspace{1mm}|\hspace{1mm}m(u)=d\}|
&=
\sum_{i=0}^{r-1}(d-1)^{i} 
=
\frac{(d-1)^{r}-1}{d-2} ,\\
|\{ u\in U\hspace{1mm}|\hspace{1mm}m(u)=1\}|
&=
(d-1)^r , \\
|\{ w\in W\hspace{1mm}|\hspace{1mm}m(w)=d\}|
&=\sum_{i=0}^{r-2}(d-1)^{i}
=
\frac{(d-1)^{r-1}-1}{d-2}\frac{(d-1)^{r}-1}{d-2} ,\\
|\{ w\in W\hspace{1mm}|\hspace{1mm}m(u)=1\}|
&=
(d-1)^{r-1} . 
\end{align*}
Thus, we have
\begin{align*}
\frac{1}{m(\emptyset )}&\sum_{x,y\in V}m(x)m(y) d(f(x),f(y))^2 \\
&=
\frac{d-2}{2d\left\{ (d-1)^{r}-1\right\}} \times
2d(p,q)^2\bigg\{
d^2 \frac{\left\{ (d-1)^{r-1}-1\right\}\left\{ (d-1)^r -1\right\}}{(d-2)^2} \\
&\hspace{5mm}+
d \frac{ (d-1)^{r-1}\left\{ (d-1)^r -1\right\}}{d-2}
+d \frac{\left\{ (d-1)^{r-1}-1\right\} (d-1)^r}{d-2}
+(d-1)^{2r-1} \bigg\} \\
&=
d(p,q)^2
\bigg\{
\frac{d\left\{ (d-1)^{r-1}-1\right\}}{d-2}
+(d-1)^{r-1} \\
&\hspace{5mm}
+\frac{\left\{ (d-1)^{r-1}-1\right\} (d-1)^r}{(d-1)^{r}-1}
+\frac{(d-2)\left\{(d-1)^{2r-1}\right\}}{d\left\{ (d-1)^{r}-1\right\}} \bigg\} .
\end{align*}
On the other hand, 
\begin{equation*}
\sum_{x,y\in V}m(x,y) d(f(x),f(y))^2
=
2 d(p,q)^2 .
\end{equation*}
Hence, 
\begin{multline*}
\frac{\sum_{x,y\in V}m(x,y) d(f(x),f(y))^2}{
\frac{1}{m(\emptyset )}\sum_{x,y\in V}m(x)m(y) d(f(x),f(y))^2} \\
=
2\bigg\{
\frac{d\left\{ (d-1)^{r-1}-1\right\}}{d-2}
+(d-1)^{r-1}+\hspace{5cm}\\
\frac{\left\{ (d-1)^{r-1}-1\right\} (d-1)^r}{(d-1)^{r}-1}
+\frac{(d-2)\left\{(d-1)^{2r-1}\right\}}{d\left\{ (d-1)^{r}-1\right\}}\bigg\}^{-1}, 
\end{multline*}
which proves the proposition. 
\end{proof}

The following corollary is straightforward from 
Proposition \ref{from-below-prop} and \ref{from-above-prop}. 

\begin{corollary}\label{tree-order}
Let $X$ be any metric space, 
and let $d\geq 3$. 
Then, we have    
$$
\lambda(T_{d,r},X)\asymp_d \frac{1}{(d-1)^r}.
$$
\end{corollary}

\begin{remark}\label{path-graph-remark}
When $d=2$, $T_{2,r}$ is just a path graph with $2r+1$ vertices. 
In this case nonlinear spectral gaps are never 
less than the linear spectral gap as we see below. 
Let $P_n=(V,E)$ be a path graph with $n+1$ vertices.  
More precisely, suppose that $V=\{ v_0 ,\ldots , v_{n}\}$ and 
$(v_i ,v_j) \in\vec{E}$ if and only if $|i-j|=1$. 
Let $m$ be an arbitrary weight function on $P_n$, and let  
$(X,d)$ be a metric space. 
For any $f: V\to X$, we define a 
map $\varphi_f : V\to \mathbb{R}$ by setting 
$$
\varphi_f (v_i )
=
\begin{cases}
0,&\quad \textrm{if}\quad i=0, \\
\sum_{1\leq l\leq i} d( f(v_{l-1}),f(v_l) )
,&\quad \textrm{if}\quad 2\leq i\leq n .
\end{cases}
$$
Then by the triangle inequality we have 
\begin{align*}
\frac{1}{m (\emptyset )}\sum_{x,y\in V}m(x)& m(y) d( f(x), f(y))^2 \\
&\leq
\frac{1}{m (\emptyset )}\sum_{x,y\in V}m(x) m(y) | \varphi_f (x)-\varphi_f (y) |^2  \\
&\leq
\frac{1}{\mu_1 (P_n )}
\sum_{x,y\in V}m(x,y) | \varphi_f (x)-\varphi_f (y) |^2 \\
&=
\frac{1}{\mu_1 (P_n )}
\sum_{x,y\in V}m(x,y)d( f(x), f(y))^2 ,
\end{align*}
which implies 
$$
\lambda (P_n , X) \geq\mu_1 (P_n ) .
$$
We remark that if the weight function $m$ is uniform, it is known that 
$$
\mu_1 (P_n ) =1-\cos\frac{\pi}{n}.
$$
\end{remark}

\begin{acknowledgement}
The authors would like to 
thank Professor Hiroyasu Izeki for 
valuable discussions concerning 
Proposition \ref{izeki-estimate}. 
We are also grateful to the anonymous referee for their careful reading of this paper and many 
helpful suggestions. 
The first author was supported by JST, CREST, ``A mathematical challenge 
to a new phase of material sciences''. 
\end{acknowledgement}

\paragraph{\textbf{Added in proof}}
After we have completed this work, we found a uniform lower estimate of the nonlinear spectral gap for a regular graph 
has been obtained by Mendel and Naor \cite[Lemma 2.1]{MN2}. 
According to their estimate we have 
$$
\lambda(H_n , X)\ge\frac{1}{n4^n}.
$$
However, this is not sharp by Proposition \ref{hamming-prop}.
Their estimate cannot be applied for $T_{d,r}$ 
since this is not a regular graph.

\end{document}